\documentclass[11pt, a4paper,pagebackref]{amsart}
\usepackage{amsmath}
\usepackage{amssymb}
\usepackage{amsthm}
\usepackage{amscd}
\usepackage{tikz-cd}

\usepackage{color}
\usepackage{hyperref}

\usepackage[inner=2.3cm,outer=2.3cm, bottom=3.3cm]{geometry}
\usepackage{setspace}

\numberwithin{equation}{section}

\usepackage[colorinlistoftodos]{todonotes}

\usepackage[all]{xy}
\usepackage{enumerate}
\usepackage{hyperref}

\usepackage{color}
\hypersetup{colorlinks=true} 
\newtheorem{theorem}{Theorem}[section]

\newtheorem{lemma}[theorem]{Lemma}
\newtheorem{proposition}[theorem]{Proposition}
\newtheorem{corollary}[theorem]{Corollary}

\theoremstyle{definition}

\newtheorem{remark}[theorem]{Remark}
\newtheorem{remark and definition}[theorem]{Remark and Definition}
\newtheorem{remark and notation}[theorem]{Remark and Notation}

\newtheorem*{setup}{Setup}

\newtheorem{example}[theorem]{Example}

\newtheorem{question}[theorem]{Question}

\newcommand\Hom{\operatorname{Hom}}
\newcommand\Ext{\operatorname{Ext}}
\newcommand\Tor{\operatorname{Tor}}

\newcommand\depth{\operatorname{depth}}
\newcommand\grade{\operatorname{grade}}

\newcommand\Ker{\operatorname{\Ker}}

\newcommand\pd{\operatorname{pd}}

\newcommand\RHom{\operatorname{\mathbf{R}Hom}}

\DeclareMathOperator{\Hdim}{H-dim}

\DeclareMathOperator{\type}{r}

\DeclareMathOperator{\gid}{Gid}

\author[Holanda,  Jorge-P\'erez, Mendoza-Rubio]{Rafael Holanda, Victor H. Jorge-P\'erez and Victor D. Mendoza-Rubio}

\address{Departamento de Matemática, CCEN, Universidade Federal de Pernambuco, Recife, PE, 50740-560, Brazil}
\email{rafael.holanda@ufpe.br}

\address{Universidade de S{\~a}o Paulo -
ICMC, Caixa Postal 668, 13560-970, S{\~a}o Carlos-SP, Brazil}
\email{vhjperez@icmc.usp.br}
\email{vicdamenru@gmail.com}

\subjclass[2020]{Primary: 13D05, 13D07, 13H10. The second author was supported by grant 2019/21181-0, São Paulo Research Foundation (FAPESP). The third author was supported by grants 2022/03372-5 and 2023/15733-5, S\~ao Paulo Research Foundation (FAPESP)}

\begin{document}

\title{On extension modules of finite homological dimension}

\maketitle

\begin{abstract}
We explore the implications of the finiteness of homological dimensions for Ext modules, focusing on projective dimension, injective dimension, and their Gorenstein counterpart. In this direction, we establish several finiteness criteria for homological dimensions. Under such finiteness conditions, our main result is a new duality for certain Ext modules when tensored by canonical modules. 
\end{abstract}

\section{Introduction}

After the astonishing ideas of Auslander, Buchsbaum, and Serre culminating in the characterization of the regular property of a Noetherian local ring by the finiteness of the projective dimension of its residual field, homological algebra has played a fundamental role in commutative ring theory.  Indeed, this foundation enabled diverse characterizations of ring-theoretic properties via finiteness conditions on homological dimensions (e.g., projective, complete intersection, Gorenstein dimensions, and their injective versions), see, for instance, some classical references \cite{Stablemoduletheory, Completeintersectiondimension, GorensteinDimensions, isomorphismsbetweencomplexeswithapplicationstothehomologicaltheoryofmodules, Dimensionprojectivefinietcohomologielocale}. Establishing criteria for the finiteness of homological dimensions as well as their prescription -- especially via vanishing of Ext modules -- is a very active topic in homological commutative algebra, containing a vast literature including but not limited to 
\cite{Stablemoduletheory, HomologicalDimensionsAndRelatedInvariantsOfModulesOverLocalRings, GorensteinDimensions,  OnGorensteinProjectiveInjectiveAndFlatDimensionsAFunctorialDescriptionWithApplications, HomologicalDimensionsTheGorensteinPropertyAndSpecialCasesOfSomeConjectures, FH, GorensteinringsviahomologicaldimensionsandsymmetryinvanishingofExtandTatecohomology, ARCForModulesWhose(Self)DualHasFiniteCompleteIntersectionDimensionv2, VanishingOf(Co)homologyfrenessCriteriaAndTheAuslanderReitenConjectureforCohen-MacaulayBurch, GdimensionOverLocalHomomorphisms, AuslanderReitenCinjetivadimensionVanishingofExt, OnModulesWhoseDualIsOfFiniteGorensteinDimension, GorensteinDimensionAndTorsionOfModulesOverCommutativeNoetherianRings, NumericalApectsofComplexesOfFiniteHomologicalDimensions,HomologicalDimensionsOfRigidModules}. 

Recently, an approach to study the implications of vanishing of Ext is under the condition that certain Hom module has some finite homological dimension, as evidenced in \cite{HomologicalDimensionsTheGorensteinPropertyAndSpecialCasesOfSomeConjectures, injectivedimensiontakahashi, OnModulesWhoseDualIsOfFiniteGorensteinDimension}. In addition, Ghosh-Puthenpurakal \cite{GorensteinringsviahomologicaldimensionsandsymmetryinvanishingofExtandTatecohomology} and  Zargar-Gheibi \cite{NumericalApectsofComplexesOfFiniteHomologicalDimensions} studied the consequences of having Ext modules with finite injective dimension, deriving criteria for finiteness for the projective dimension of modules. Motivated by all, in this paper, we aim to study the consequences of having Ext modules with some finite homological dimension, focusing on the projective dimension, injective dimension, and their Gorenstein versions. Thus, the standard assumption in the present text does not revolve around the usual Ext vanishing hypothesis but rather the finiteness of its homological dimensions. 

Zargar and Gheibi proved in \cite{NumericalApectsofComplexesOfFiniteHomologicalDimensions} that for non-zero finitely generated modules $M$ and $N$ over a Noetherian local ring $R$ such that $\operatorname{id}_R(\Ext_R^i(M,N))<\infty$ for all $i$, the finiteness of the projective dimension of $M$ is equivalent to that of injective dimension of $N$. Then, what can one tell if one considers the Ext modules $\Ext_R^{i\geq 0}(M,N)$ with another finite homological dimension under the additional condition that $\operatorname{pd}_R(M)<\infty$ or $\operatorname{id}_R(N)<\infty$? We address this question in the following theorem.

\begin{theorem}[{\rm Corollary \ref{SPA11} \& Corollary \ref{sqly}}]\label{main1}
    Let $R$ be a Noetherian  local ring of dimension $d$, and let $M$ and $N$ be non-zero finitely generated $R$-modules. Let $\mathcal{H}$ be any of the homological dimensions $\operatorname{pd}, \operatorname{G-dim}, \operatorname{Gid}$, and suppose that  $\mathcal{H}(\Ext_R^i(M,N))<\infty$ for all $0\leq i \leq d-\operatorname{depth} (M)$.
    \begin{enumerate}[ \rm (1)]
        \item Suppose that $\operatorname{pd}_R (M)<\infty$.
        \begin{enumerate}[\rm ({1}.a)]
            \item If $\mathcal{H}\in  \{\operatorname{pd},\operatorname{G-dim}\}$, then $\mathcal{H}(N)<\infty$. 
            \item If $R$ has a dualizing complex and $\mathcal{H}=\operatorname{Gid}$, then $\operatorname{Gid}_R(N)<\infty$.  
        \end{enumerate}
        \item Suppose that $\operatorname{id}_R (N)<\infty$.
        \begin{enumerate}[\rm ({2}.a)]
            \item If $\mathcal{H}= \operatorname{pd}$, then $\operatorname{id}_R(M)<\infty$. 
            \item If $R$ has a dualizing complex and $\mathcal{H}=\operatorname{G-dim}$, then $\operatorname{Gid}_R (M)<\infty$. 
            \item If $\mathcal{H}=\operatorname{Gid}$, then $\operatorname{G-dim}_R(M)<\infty$.  
        \end{enumerate}
    \end{enumerate}
\end{theorem}

Theorem \ref{main1} not only gives criteria for detecting the finiteness of homological dimension of modules but also allows one to obtain Gorenstein criteria as verified in Corollary \ref{Gor1} and Corollary \ref{Gor2}. In addition, it is worth noting that Theorem \ref{main1} says that the counterparts with Gorenstein dimensions of the results of Foxby in  \cite[Proposition 3.5(b)]{onthemuiinaminimalinjectiveresolution} are true. Furthermore, item (2.a) allows us to recover a celebrated result of Foxby given in \cite[Corollary 4.4]{isomorphismsbetweencomplexeswithapplicationstothehomologicaltheoryofmodules}, and item (2.c) gives a positive answer to the Gorenstein version of the question in \cite[Question 4.5]{GorensteinringsviahomologicaldimensionsandsymmetryinvanishingofExtandTatecohomology}.

In our study, we also consider deficiency modules, which were introduced by Schenzel \cite{Sch98}. Fiel and Holanda \cite{FH} proved that the finiteness of the projective (resp. injective) dimension of all deficiency modules of a finitely generated module $M$ over a local ring $R$ implies that $M$ has finite injective (resp. projective) dimension.  We not only prove this result with a different approach, but also, we establish its Gorenstein version. The converse is assured for Cohen-Macaulay modules, extending Foxby \cite{onthemuiinaminimalinjectiveresolution}.

\begin{theorem}[{\rm Theorem \ref{sq18z}}]\label{defmod}
    Let $R$ be a Noetherian local ring with a dualizing complex and let $M$ be a finitely generated $R$-module.
    \begin{enumerate}[\rm (1)]
        \item If $\operatorname{G-dim}_R(K^i(M))<\infty$ for all $i\geq 0$, then $\operatorname{Gid}_R(M)<\infty$.
        \item If $\operatorname{Gid}_R(K^i(M))<\infty$ for all $i\geq 0$, then $\operatorname{G-dim}_R(M)<\infty$.
    \end{enumerate}
\end{theorem}

Our main result is a new duality concerning the canonical module when tensored by certain Ext modules. Such a duality is also established for semidualizing modules (Theorem \ref{Cdualt}) and is applied in many directions throughout the paper.

\begin{theorem}[{\rm Theorem \ref{sqmm}}]\label{sqmm1}
    Assume $R$ is a Cohen-Macaulay local with canonical module $\omega$, and let $n$ be a non-negative integer. Let $\operatorname{H-dim}_R$ denote either 
    $\operatorname{pd}$ or $\operatorname{G-dim}$,  and $\operatorname{Hid}_R$ the injective counterpart. Let $M$ and $N$ be finitely generated $R$-modules such that $\operatorname{H-dim}_R(N)<\infty$.  Then the following conditions are equivalent:
    \begin{enumerate}[\rm(1)]
        \item $\operatorname{H-dim}_R(\Ext_R^i(M,N))<\infty$ for all $0\leq i \leq n$.
        \item $\operatorname{Hid}_R(\Ext_R^i(M, N \otimes_R \omega))<\infty$ for all $0\leq i \leq n$.
    \end{enumerate}
   In any case, we have natural isomorphisms $$\Ext_R^i(M,N) \otimes_R \omega \cong \Ext_R^i(M, N \otimes_R \omega)$$
    for all $0\leq i \leq n$. 
\end{theorem}

Using the behavior of the canonical module described in Theorem \ref{sqmm1}, we derive the following:

\begin{corollary}[{Corollary \ref{qza}}]
    Let $R$ be  a Cohen-Macaulay local ring of depth $t$ and, $M$ and $N$ be non-zero finitely generated $R$-modules such that $\operatorname{pd}_R(\Ext_R^i(M,N))<\infty$ for all $0\leq i \leq t-\operatorname{depth}(M)$, then $\operatorname{pd}_R(M)<\infty$ if and only if $\operatorname{pd}_R(N)<\infty$.
\end{corollary}

The organization of this paper is as follows.  In Section \ref{sectionextofcomplexes}, we establish results concerning Ext modules of complexes of finite homological dimension, and prove Theorem \ref{main1}, and derive Gorenstein criteria. In Section \ref{sectiondeficiencymodules}, we study the implications of having deficiency modules with some finite homological dimension. In Section  \ref{sectionduality}, we prove our main result, namely Theorem \ref{sqmm1}, and derive several consequences. In Section \ref{sectionsquestions}, we discuss some questions motivated by results in Section \ref{sectionduality}. Finally, in Section \ref{sectionexamples}, we present some examples of Ext modules with finite homological dimension. Now, we establish our setup.

\begin{setup} Throughout this paper,  $R$ stands for a Noetherian (commutative unitary) ring. Whenever we assume that $R$ is local, its maximal ideal will be denoted by $\mathfrak{m}$ and its residue field by $k$. The depth of $R$ is denoted by $t$. In that case, for a finitely generated $R$-module $M$, its $i$-th syzygy module and its transpose of Auslander are considered defined by a minimal free resolution of $M$, and denoted by $\Omega^i(M)$ and  $\operatorname{Tr}_R(M)$, respectively. 

We also consider the derived category of $R$. Unless, otherwise specified, we adopt the terminology used in 
\cite{DerivedCategoryMethodsInCommutativeAlgebra}. In  particular, we denote the projective and Gorenstein dimensions of an $R$-module $M$ respectively by $\operatorname{pd}_R(M)$ and  $\operatorname{G-dim}_R(M)$, while the injective dimension and Gorenstein injective dimension of $M$ are respectively denoted by $\operatorname{id}_R(M)$ and $\operatorname{Gid}_R(M)$. In addition, for an $R$-module $M$, we let $\operatorname{H-dim}_R(M)$ stand for $\operatorname{pd}_R(M)$ or $\operatorname{G-dim}_R(M)$, and $\operatorname{Hid}_R(M)$ for the corresponding injective version.

If $\boldsymbol{X}$ and $\boldsymbol{Y}$ are complexes in $\mathcal{D}(R)$, we set 
\begin{center}
    $\Ext_R^i(\boldsymbol{X}, \boldsymbol{Y})=\operatorname{H}_{-i}(\RHom_R(\boldsymbol{X}, \boldsymbol{Y}))$ and $\Tor_i^R(\boldsymbol{X}, \boldsymbol{Y})=\operatorname{H}_i(\boldsymbol{X} \otimes_R^{\mathbf{L}} \boldsymbol{Y}), \text{ for all } i.$
\end{center}

\end{setup}

\section{Finite homological dimensions of Ext  of complexes}\label{sectionextofcomplexes}

In this section, we derive results regarding the finiteness of homological dimension of Ext of complexes and conclude the corresponding results for Ext of modules.

\begin{lemma} \label{lemGid} \label{lemHdim}
    Let $ \boldsymbol{X}$
    be a complex in $\mathcal{D}_{\square}(R).$
    \begin{enumerate}[\rm(1)]
        \item If $\gid_R(\operatorname{H}_i(\boldsymbol{X}))<\infty$ for all $i$, then $\gid_R(\boldsymbol{X})<\infty$.
        \item If $\operatorname{H-dim}_R(\operatorname{H}_i(\boldsymbol{X}))<\infty$ for all $i$, then $\operatorname{H-dim}_R(\boldsymbol{X})<\infty$. 
    \end{enumerate}
\end{lemma}
\begin{proof} We prove item (1). This is done by induction on $s=\sup \boldsymbol{X}-\inf \boldsymbol{X}$. If $s=0$, then $\boldsymbol{X}\simeq \Sigma^j \operatorname{H}_j(\boldsymbol{X})$ for some $j$, and the conclusion follows. 

Assume $s>0$ and let $u=\sup \boldsymbol{X}$. Note that there is a natural short exact sequence of complexes $0\to \boldsymbol{Y} \to \boldsymbol{X} \to \tau_{_{\subset u-1}}(\boldsymbol{X})\to 0, $ where $\tau_{_{\subset u-1}}(\boldsymbol{X})$ denotes the $(t-1)$ soft left truncation of $\boldsymbol{X}$. In addition, note 
    $\boldsymbol{X} \to \tau_{_{\subset u-1}}(\boldsymbol{X})$ induces an isomorphism in homologies in degree $\leq u-1$, and that $\operatorname{H}_i(\tau_{_{\subset u-1}}(\boldsymbol{X}))=0$ for $i>u-1$. Also, for $i \not=u$, we have that $\operatorname{H}_i(\boldsymbol{Y})=0$ and $\operatorname{H}_u(\boldsymbol{Y})\cong \operatorname{H}_u(\boldsymbol{X})$. 

     If $\gid_R(\operatorname{H}_i(\boldsymbol{X}))<\infty$ for all $i$, note that $\boldsymbol{Y}$ and $\tau_{_{\subset t-1}(\boldsymbol{X})}$ satisfy the induction hypothesis, and therefore they have finite Gorenstein injective dimension. Thus, the conclusion follows from the fact that the class of complexes in $\mathcal{D}(R)$ with finite Gorenstein injective dimension has the two out three property \cite[Proposition 9.2.15]{DerivedCategoryMethodsInCommutativeAlgebra}.

 The proof of item (2) is similar using  \cite[Corollary 8.1.9 and Proposition 9.1.16]{DerivedCategoryMethodsInCommutativeAlgebra}.
\end{proof}

\begin{theorem}\label{GSPA11}
    Assume $R$ is local, and let $\boldsymbol{M}$ and $\boldsymbol{N}$ be non-acyclic complexes in $\mathcal{D}_{\square}^{\rm f}(R)$ such that  $\operatorname{pd}_R(\boldsymbol{M})<\infty$.
    \begin{enumerate}[\rm (1)]
        \item   If $\Hdim_R(\Ext_R^i(\boldsymbol{M}, \boldsymbol{N}))<\infty$  for all $i$, then $\operatorname{H-dim}_R(\boldsymbol{N})<\infty$.
        \item  Suppose that $R$ has a dualizing complex.  If $\operatorname{Gid}_R(\Ext_R^i(\boldsymbol{M}, \boldsymbol{N}))<\infty$ for all $i$, then $\operatorname{Gid}_R(\boldsymbol{N})<\infty$.
    \end{enumerate}
\end{theorem}
\begin{proof}
     We know that $\Ext_R^i(\boldsymbol{M},\boldsymbol{N})=\operatorname{H}_{-i}(\RHom_R(\boldsymbol{M},\boldsymbol{N}))$ for all $i$. From \cite[A.4.4 and A.5.4]{GorensteinDimensions}, one sees that $\RHom_R(\boldsymbol{M},\boldsymbol{N}) \in \mathcal{D}_{\square}^{\rm f}(R)$. By Lemma \ref{lemHdim}(2) (resp.  Lemma \ref{lemGid}(1)), we conclude  $\operatorname{H-dim}_R(\RHom_R(\boldsymbol{M}, \boldsymbol{N}))<\infty$ (resp. \linebreak $\operatorname{Gid}_R(\RHom_R(\boldsymbol{M}, \boldsymbol{N}))<\infty$).  It follows from \cite[Theorems 5.6 \& 5.8]{GdimensionOverLocalHomomorphisms} (resp. \cite[Proposition 4.1(b)]{CompleteIntersectionDimensionsandFoxbyClasses} and \cite[Theorem 4.4]{OnGorensteinProjectiveInjectiveAndFlatDimensionsAFunctorialDescriptionWithApplications})  that $\operatorname{H-dim}_R (\boldsymbol{N})<\infty$ (resp. $\operatorname{Gid}_R (\boldsymbol{N})<\infty$), concluding (1) (resp. (2)). 
    \end{proof}

\begin{corollary} \label{SPA11} 
    Assume $R$ is local, and let $M$ and $N$ be non-zero finitely generated $R$-modules such that $\operatorname{pd}_R(M)<\infty$. 
    \begin{enumerate}[\rm(1)]
        \item If $\operatorname{H-dim}_R(\Ext_R^i(M,N))<\infty$ for all $0\leq i \leq \operatorname{pd}_R(M)$, then $\operatorname{H-dim}_R(N)<\infty$. 
        \item Suppose that $R$ has a dualizing complex. If $\operatorname{Gid}_R(\Ext_R^i(M,N))<\infty$ for all $0\leq i \leq \operatorname{pd}_R(M)$, then $\operatorname{Gid}_R(N)<\infty$. 
    \end{enumerate}
    \end{corollary}

 \begin{corollary}\label{Gor1}
     Assume $R$ is local, and let $M$ be a non-zero finitely generated $R$-module such that $\operatorname{pd}_R(M)<\infty$ and $\operatorname{Gid}_R(\Ext_R^i(M,M))<\infty$ for all $0\leq i\leq t-\operatorname{depth}(M)$.  Then $R$ is Gorenstein. 
 \end{corollary}  
 \begin{proof}
     We may assume that $R$ is complete. Then $R$ has a dualizing complex. By Corollary \ref{SPA11}(2), then $\operatorname{Gid}_R(M)<\infty$. Then the result follows from  \cite[Theorem 6.3.2]{GorensteinDimensions}.
 \end{proof}

Item (2) of the next theorem provides a positive answer to the Gorenstein version of  \cite[Question 4.5]{GorensteinringsviahomologicaldimensionsandsymmetryinvanishingofExtandTatecohomology}. 

\begin{theorem}\label{q221}
   Assume $R$ is local, and let $\boldsymbol{M}$ and $\boldsymbol{N}$ be non-acyclic complexes in $\mathcal{D}_{\square}^{\rm f}(R)$. 
    Suppose that $\operatorname{id}_R(\boldsymbol{N})<\infty$. In the case $\operatorname{H-dim}=\operatorname{G-dim}$, assume further that $R$ has a dualizing complex.
    \begin{enumerate}[\rm(1)]
        \item If $\operatorname{H-dim}_R(\Ext_R^i(\boldsymbol{M},\boldsymbol{N}))<\infty$ for all $i$, then $\operatorname{Hid}_R(\boldsymbol{M})<\infty$.  
        \item If $\operatorname{Gid}_R(\Ext_R^i(\boldsymbol{M}, \boldsymbol{N}))<\infty$ for all $i$, then  $\operatorname{G-dim}_R(\boldsymbol{M})<\infty$.
    \end{enumerate}
\end{theorem}

\begin{proof}
    We know that $\Ext_R^i(\boldsymbol{M},\boldsymbol{N})=\operatorname{H}_{-i}(\mathbf{R}\mathrm{Hom}_R(\boldsymbol{M},\boldsymbol{N}))$. Additionally, by \cite[A.4.4 and A.5.2]{GorensteinDimensions}, we see that $\RHom_R(\boldsymbol{M},\boldsymbol{N}) \in \mathcal{D}_{\square}^{\rm f}(R)$. In the case (1), for all $i$, $\operatorname{H-dim}_R(\operatorname{H}_i(\mathbf{R}\mathrm{Hom}_R(\boldsymbol{M},\boldsymbol{N})))<\infty$.  Thus,  by Lemma \ref{lemGid}(1), we have that $\operatorname{H-dim}_R(\RHom_R(\boldsymbol{M}, \boldsymbol{N}))<\infty$. As to (2), one considers Lemma \ref{lemHdim}(2) instead so that $\operatorname{Gid}_R(\RHom_R(\boldsymbol{M}, \boldsymbol{N}))<\infty$.
     
The conclusion of item (1) for 
$\operatorname{H-dim}_R=\operatorname{pd}_R$ follows from \cite[Lemma 6.2.12]{HyperhomologicalAlgebra}.  

It remains to prove  (1) for $\operatorname{H-dim}_R=\operatorname{G-dim},$ and (2).  For proving both, either by \cite[2.4]{GdimensionOverLocalHomomorphisms}, or by assumption, \cite[Proposition 5.9]{Semi-dualizingComplexesAndTheirAuslanderCategores} and \cite[Theorem 19.2.34]{DerivedCategoryMethodsInCommutativeAlgebra}, we can let $\boldsymbol{D}$ be a dualizing complex of $R$. Since $\operatorname{id}_R(\boldsymbol{N})<\infty$, then $\boldsymbol{N} \in \mathcal{B}(R)$  and hence  $\operatorname{pd}_R(\RHom_R(\boldsymbol{D},\boldsymbol{N}))<\infty$ by the Foxby equivalence. 
    
%    We have the following isomorphisms in $\mathcal{D}(R)$:  $$\RHom_R(M,N) \simeq \RHom_R(M, \omega \otimes_R^{\mathbf{L}} \RHom_R(\omega, N)) \simeq  \RHom_R(M, \omega) \otimes_R^{\mathbf{L}} \RHom_R(\omega,N),$$
 %   where the last isomorfism is by  \cite[A.4.23]{GorensteinDimensions}
 We have
    \begin{align*}
   \RHom_R(\boldsymbol{M},\boldsymbol{N})& \simeq \RHom_R(\boldsymbol{M}, \boldsymbol{D} \otimes_R^{\mathbf{L}} \RHom_R(\boldsymbol{D}, \boldsymbol{N})) && \text{because } \boldsymbol{N} \in \mathcal{B}(R)\\ & \simeq  \RHom_R(\boldsymbol{M}, \boldsymbol{D}) \otimes_R^{\mathbf{L}} \RHom_R(\boldsymbol{D},\boldsymbol{N})&& \text{by  \cite[Theorem 4.4.5(a)]{HyperhomologicalAlgebra}}.
\end{align*}

In the case (1) for $\operatorname{H-dim}_R=\operatorname{G-dim}_R$, we have that $\operatorname{G-dim}_R(\RHom_R(\boldsymbol{M},\boldsymbol{N}))<\infty$, concluding from \cite[Theorems 5.1]{GdimensionOverLocalHomomorphisms} that $\operatorname{G-dim}_R(\RHom_R(\boldsymbol{M},\boldsymbol{D}))<\infty$. Hence $\operatorname{Gid}_R(\boldsymbol{M})<\infty$  by \cite[Corollary 6.4]{OnGorensteinProjectiveInjectiveAndFlatDimensionsAFunctorialDescriptionWithApplications}.

As to (2), we have that $\operatorname{Gid}_R(\RHom_R(\boldsymbol{M}, \boldsymbol{N}))<\infty$. It follows from 
 the last isomorphism, \cite[Proposition 4.1(b)]{CompleteIntersectionDimensionsandFoxbyClasses} and \cite[Theorem 4.4]{OnGorensteinProjectiveInjectiveAndFlatDimensionsAFunctorialDescriptionWithApplications} that  \linebreak $\operatorname{Gid}_R(\RHom_R(\boldsymbol{M}, \boldsymbol{D}))<\infty$. Since $M \simeq \RHom_R(\RHom_R(M,\boldsymbol{D}), \boldsymbol{D})$ by \cite[A.8.5]{GorensteinDimensions}, then it follows from \cite[Corollary 6.4]{OnGorensteinProjectiveInjectiveAndFlatDimensionsAFunctorialDescriptionWithApplications} that $\operatorname{G-dim}_R(\boldsymbol{M})<\infty$.
\end{proof}

\begin{corollary}\label{sqly}
    Assume $R$ is Cohen-Macaulay local, and let $M$ and $N$ be non-zero finitely generated $R$-modules  such that  $\operatorname{id}_R(N)<\infty$. In the case $\operatorname{H-dim}_R=\operatorname{G-dim}_R$, assume further that $R$ has canonical module. Then
    \begin{enumerate}[\rm(1)]
        \item If $\operatorname{H-dim}_R(\Ext_R^i(M,N))<\infty$ for all $0\leq i \leq t-\operatorname{depth} (M)$, then $\operatorname{Hid}_R(M)<~\infty$.
        \item If $\operatorname{Gid}_R(\Ext_R^i(M,N))<\infty$ for all $0\leq i \leq t-\operatorname{depth} (M)$, then $\operatorname{G-dim}_R(M)<~\infty$.
    \end{enumerate}
\end{corollary}

When $\operatorname{H-dim}_R=\operatorname{pd}_R$, note that Corollary \ref{sqly}(1) recovers the celebrated theorem of Foxby in \cite[Corollary 4.4]{isomorphismsbetweencomplexeswithapplicationstothehomologicaltheoryofmodules}, which states that if there exists a (finitely generated) non-zero $R$-module $M$ with both finite injective dimension and projective dimension, then $R$ is Gorenstein. On the other hand,  via Corollary \ref{sqly}(2), we obtain the following Gorenstein criterion that generalizes \cite[Corollary 4.7]{GorensteinringsviahomologicaldimensionsandsymmetryinvanishingofExtandTatecohomology}
 and \cite[Corollary 4.4]{NumericalApectsofComplexesOfFiniteHomologicalDimensions}.
\begin{corollary}\label{Gor2}
        Assume $R$ is Cohen-Macaulay local, and let $M$ be a non-zero finitely generated $R$-module such that $\operatorname{id}_R(M)<\infty$ and $\operatorname{Gid}_R(\Ext_R^i(M,M))<\infty$ for all $0\leq i \leq t-\operatorname{depth}(M)$. Then $R$ is Gorenstein. 
\end{corollary}
\begin{proof}
    By Corollary \ref{sqly}(2), $\operatorname{G-dim}_R(M)<\infty$. Then the result follows from \cite[Theorem 6.3.2]{GorensteinDimensions}. 
\end{proof}

\section{Finite Homological dimensions of deficiency modules}\label{sectiondeficiencymodules}

\begin{setup}
In the rest of the paper, $R$ is local, and unless otherwise stated, all $R$-modules will be finitely generated. Moreover, in this section, we assume that $R$ has a dualizing complex $\boldsymbol{D}$. Also, we assume that $\boldsymbol{D}$ is normalized in the sense that $\sup D=\dim R$. 
\end{setup}

The existence of $\boldsymbol{D}$ is equivalent to $R$ being a factor of a Gorenstein local ring $S$, i.e., $R$ is a homomorphic image of $S$. In that context, Schenzel \cite{Sch98} generalized the notion of a canonical module of a ring in the following sense: For an $R$-module $M$, set 
$$K^i(M):=\Ext_S^{\dim (S)-i}(M,S)$$
for all $i \in \mathbb{Z}$. For each $i$, the $R$-module $K^i(M)$ is called the $i$-th \textit{deficiency module} of $M$. The $R$-module $K(M):=K^{\dim_R(M)}(M)$ is called the \textit{canonical module} of $M$. By local duality we have $K^i(M)=0$ for $i<\depth(M)$ or $i>\dim(M)$, and $K^i(M)\neq0$ for $i=\depth (M)$ and $i=\dim (M)$. The modules of deficiency also satisfy the following isomorphism for all $i$:
\begin{equation*}\label{112}
    K^i(M) \cong  \operatorname{H}_{i}(\RHom_R(M, \boldsymbol{D})).
\end{equation*}
For further details about deficiency modules, we refer the reader to \cite{Sch98}.

 In \cite{FH}, the authors investigated the consequences of the finiteness of classical homological dimensions, namely, projective and injective dimensions, of the deficiency modules.  Motivated by their work, we apply Theorem~\ref{q221} to extend those results to Gorenstein dimensions. We also emphasize that our approach differs from theirs in both method and perspective.

\begin{theorem}\label{sq18z}
    Let $M$ be an $R$-module. Then
    \begin{enumerate}[\rm(1)]
        \item If $\operatorname{H-dim}_R(K^i(M))<\infty$ for all $\operatorname{depth}(M)\leq i \leq \dim(M)$, then $\operatorname{Hid}_R(M)<\infty$.
        \item If $\operatorname{Hid}_R(K^i(M))<\infty$ for all $\operatorname{depth}(M)\leq i \leq \dim(M)$, then $\operatorname{H-dim}_R(M)<\infty$.
    \end{enumerate}
\end{theorem}
\begin{proof}

For each $i$, we note that $K^i(M)=\operatorname{H}_{i}(\RHom_R(M, \boldsymbol{D}))=\Ext_R^{-i}(M, \boldsymbol{D})$. Since $K^i(M)=0$ for $i<\depth (M)$ or $i>\operatorname{dim}_R(M)$, by assumption we see have that all modules $\Ext_R^i(M,\boldsymbol{D})$ have finite $\operatorname{H-dim}_R$ or all modules have finite $\operatorname{Hid}$. Then the result follows from \cite[Proposition 4.2]{NumericalApectsofComplexesOfFiniteHomologicalDimensions} and Theorem \ref{q221}. 
\end{proof}

%Consequently, we ascertain Foxby's equivalence \cite[Corollary 3.6]{onthemuiinaminimalinjectiveresolution} for Gorenstein (injective) dimension in Corollary \ref{wqfa}, and prove that different characterizations of the Gorenstein property of a Cohen-Macaulay ring come from the same place, Corollary \ref{gorequivalences}

The following corollary contains the counterparts of \cite[Proposition 3.5(b)]{onthemuiinaminimalinjectiveresolution} to Gorenstein dimensions. 
\begin{corollary}\label{wqfa}
    Let $M$ be an $R$-module. If $M$ is Cohen-Macaulay, then
    \begin{enumerate}[\rm(1)]
        \item $\operatorname{H-dim}_R(K(M))<\infty$ if and only if $\operatorname{Hid}_R(M)<\infty$.
        \item $\operatorname{H-dim}_R(M)<\infty$ if and only if $\operatorname{Hid}_R(K(M))<\infty$. 
    \end{enumerate}
\end{corollary}
\begin{proof}
 In view of \cite[Theorem 1.14]{Sch98}, it is enough to prove (1). One implication follows from Theorem \ref{sq18z}(1). For the other, suppose that $\operatorname{Hid}_R(M)<\infty$. Since $M \cong K(K(M))$, then $K(M)$ is Cohen-Macaulay (\cite[Theorem 1.14]{Sch98} again) and from Theorem \ref{sq18z}(2) it follows that $\operatorname{H-dim}_R(K(M))<\infty$. 
\end{proof}

\begin{corollary}\label{newGor}
    Let $R$ be a Cohen-Macaulay local ring with canonical module $\omega$. Assume that $\Ext_R^i(\omega, R)=0$ for all $1\leq i \leq t$. Let $\_^\dagger$ stand for $\Hom_R(\_,\omega)$.  If  $\operatorname{Gid}_R((\omega^*)^\dagger)<\infty$, then $R$ is Gorenstein. 
\end{corollary}
\begin{proof}
Note that from the vanishing hypothesis and \cite[Theorem 2.1(i)]{VanishingOf(Co)homologyfrenessCriteriaAndTheAuslanderReitenConjectureforCohen-MacaulayBurch} we see that $\omega^*$ is maximal Cohen-Macaulay. Since  $\operatorname{Gid}_R((\omega^*)^\dagger)<\infty$, then by Corollary \ref{wqfa}, $\operatorname{G-dim}_R(\omega^\ast)<\infty$. Therefore, the conclusion follows from \cite[Corollary 5.13]{OnModulesWhoseDualIsOfFiniteGorensteinDimension}. 
\end{proof}

\section{A duality for Ext modules with finite homological dimension}

\label{sectionduality}

In this section, we establish a duality of certain Ext modules. The following theorem will be often used in this section.

\begin{proposition}[{\cite[Theorem 4.4]{OnGorensteinProjectiveInjectiveAndFlatDimensionsAFunctorialDescriptionWithApplications}, \cite[Corollary 3.1.11 \& Theorem 3.4.11]{GorensteinDimensions}
}] \label{GdimCrs} 
    Assume $R$ is Cohen-Macaulay with canonical module $\omega$, and let $M$ be an $R$-module. 
    \begin{enumerate}[\rm(1)]
        \item $\operatorname{G-dim}_R(M)<\infty$ if and only if $M \in \mathcal{A}(R)$. 
        \item $\operatorname{Gid}_R(M)<\infty$ if and only if $M \in \mathcal{B}(R)$.
        \item $\operatorname{H-dim}_R(\Hom_R(\omega, M))<\infty$ if and only if $\operatorname{Hid}_R(M)<\infty$.
        \item $\operatorname{H-dim}_R(M)<\infty$ if and only if $\operatorname{Hid}_R(M\otimes_R \omega)<\infty$. 
    \end{enumerate}
\end{proposition}

In \cite[Corollary 42(3)]{HomologicalDimensionsTheGorensteinPropertyAndSpecialCasesOfSomeConjectures}, although the authors provide this without using Proposition \ref{GdimCrs}, this proposition assures that their vanishing assumptions are not necessary.

\begin{corollary}\label{85aq}
    Assume $R$ is Cohen-Macaulay with a canonical module $\omega$, and let $N$ be a non-zero $R$-module. Then $R$ is Gorenstein in any of the following cases:
    \begin{enumerate}[\rm(1)]
        \item  $\operatorname{pd}_R(N)<\infty$ and $\operatorname{G-dim}_R(\Hom_R(\omega, N))<\infty$.
        \item $\operatorname{G-dim}_R(N)<\infty$ and  $\operatorname{pd}_R(\Hom_R(\omega, N))<\infty$.
    \end{enumerate}
    \end{corollary}
 \begin{proof} Proposition \ref{GdimCrs} says that in (1), we have $\operatorname{Gid}_R(N)<\infty$ while in (2), we have $\operatorname{id}_R(N)<\infty$. The result follows from \cite[Theorem 6.3.2]{GorensteinDimensions}. 
 \end{proof}

\begin{proposition}\label{lemx}
   Assume $R$ is Cohen-Macaulay with canonical module $\omega$, and let $n$ be a non-negative integer. Let $M$ and $N$ be $R$-modules such that $\operatorname{G-dim}_R(N)<\infty$. Then \linebreak  $\operatorname{H-dim}_R(\Hom_R(M,N))<\infty$  if and only if $\operatorname{Hid}_R(\Hom_R(M, N \otimes_R \omega))<\infty$. In this case, the natural map $$\Phi_{M,N}: \Hom_R(M,N) \otimes_R  \omega \to \Hom_R(M, N \otimes_R 
    \omega)$$ is an isomorphism.
\end{proposition}

\begin{proof}
   Suppose that $\operatorname{H-dim}_R(\Hom_R(M,N))<\infty$. In particular, $\operatorname{G-dim}(\Hom_R(M,N))<\infty$. Note that  $\Hom_R(M, N)$ and $N$ belong to $\mathcal{A}(R)$. Then by \cite[Lemma 5.1]{AuslanderReitenCinjetivadimensionVanishingofExt}, the map $\Phi_{M,N}$ is an isomorphism. Hence,  it follows from Proposition \ref{GdimCrs}(3) that $\operatorname{Hid}_R(\Hom_R(M, N \otimes_R \omega))<\infty$. 

   Conversely,  suppose that $\operatorname{Hid}_R(\Hom_R(M, N \otimes_R  
 \omega))<\infty$. Then by Proposition~\ref{GdimCrs}(3), \linebreak 
 $ \operatorname{H-dim}_R(\Hom_R(\omega, \Hom_R(M, N \otimes_R \omega))<\infty$. 
   Moreover, 
   \begin{align*}
       \Hom_R( \omega, \Hom_R(M,N \otimes_R \omega))&\cong \Hom_R(\omega \otimes_R M, N \otimes_R \omega)\\ &\cong \Hom_R(M, \Hom_R(\omega, N \otimes_R  \omega))\\& \cong \Hom_R(M, N), 
   \end{align*}
   where the last isomorphism follows since $N \in \mathcal{A}(R).$ Thus, $\operatorname{H-dim}_R(\Hom_R(M,N))<\infty$. 
\end{proof}

Let $\mathcal{A}$ be an abelian category and let $\mathcal{X}$ be a class of objects in $\mathcal{A}$. We say that $\mathcal{X}$ is \textit{closed under short exact sequences} if 
the two out of three property
holds for short exact sequences of objects in $\mathcal{X}$. We say that $\mathcal{X}$ is \textit{closed under isomorphisms} if every object isomorphic to an object in $\mathcal{X}$ is in $\mathcal{X}$.

\begin{lemma}\label{lemswger}
Let $\mathcal{X}$ be a class of the category of  $R$-modules closed under short exact sequences and isomorphisms. Let $M$ and $N$ be $R$-modules, and $n$ be a non-negative integer. Suppose that $\Ext_R^i(M,N)$ and $N$ belong to $\mathcal{X}$ for all $0\leq i \leq n$. Then $\Hom_R( \Omega^i(M), N) \in \mathcal{X}$ for all $0\leq i \leq n$. 
\end{lemma}
\begin{proof}
    We must show that for all $0\leq j \leq n$, we have  $\Hom_R(\Omega^j(M), N)) \in \mathcal{X}$. We prove this by induction on $j$. For $j=0$, this holds by hypothesis. Now suppose  $0<j\leq n$ and that $\Hom_R(\Omega^{j-1}(M), N))\in\mathcal{X}$. We have an exact sequence
    $\xymatrix@=1em{0 \ar[r] & \Omega^{j}(M) \ar[r] & R^{n_j} \ar[r] & \Omega^{j-1}(M) \ar[r] & 0,}  $
    which yields 
    $$\xymatrix@=1em{
    0\ar[r] & \Hom_R(\Omega^{j-1}(M),N)\ar[r] & N^{n_j}\ar[r] & \Hom_R(\Omega^j(M),N)\ar[r] & \Ext^j_R(M,N)\ar[r] & 0
    }$$
    since $\Ext^1_R(\Omega^{j-1}(M),N)\simeq\Ext^j_R(M,N)$. The result follows by breaking this exact sequence into two short ones and using the fact that $\mathcal{X}$ is closed under short exact sequences and isomorphisms.
    
%    { \small
% \[ 0 \rightarrow \Hom_R(\Omega^{j-1}(M),N) \rightarrow N^{n_j} \rightarrow \Hom_R(\Omega^j(M), N) \rightarrow \Ext_R^1(\Omega^{j-1}(M),N) \cong  
%  \Ext_R^{j}(M,N)\rightarrow 0.\]
% }
 %It induces two exact sequences, 

\end{proof}

The following theorem improves \cite[Lemma 5.1]{AuslanderReitenCinjetivadimensionVanishingofExt}. For nomenclature on semidualizing modules and their corresponding Auslander classes, we refer to \cite{SemidualizingModules}.

\begin{theorem}\label{Cdualt} 
    Let $C$ be a semidualizing $R$-module, and $n$ be a non-negative integer. Let $M$ and $N$ be $R$-modules such that $N\in\mathcal{A}_C(R)$ and  $\Ext_R^i(M,N)\in\mathcal{A}_C(R)$ for all $0\leq i \leq n$. Then, for all $0\leq i \leq n$, there is a natural isomorphism
    $$\Ext_R^i(M,N) \otimes_R C \cong \Ext_R^i(M, N \otimes_R C).$$
\end{theorem}
\begin{proof}
For  $i=0$, the isomorphism holds from  \cite[Lemma 5.1]{AuslanderReitenCinjetivadimensionVanishingofExt}. For $i>0$, consider the short exact sequence 
$\xymatrix@=1em{
0\ar[r] & \Omega^i(M)\ar[r] & R^{r_i}\ar[r] & \Omega^{i-1}(M)\ar[r] & 0.
}$

Such an exact sequence yields the following commutative diagram
\[ \tiny{\xymatrix@C=1em{
0 \ar[r] & \Hom_R(\Omega^{i-1}(M), N) \otimes C \ar[r] \ar[d]_{\Phi_{\Omega^{i-1}(M), N}} & \Hom_R(R^{r_i}, N) \otimes C \ar[r] \ar[d]_{\Phi_{R^{r_i}, N}} & \Hom_R(\Omega^i(M),N) \otimes C \ar[r] \ar[d]_{\Phi_{\Omega^{i}(M), N}} & \operatorname{Ext}^1_R(\Omega^{i-1}(M),N) \otimes C \ar[r]  & 0 \\
0 \ar[r] & \operatorname{Hom}(\Omega^{i-1}(M), N \otimes_R C) \ar[r] & \operatorname{Hom}(R^{r_i}, N \otimes_R C) \ar[r] & \operatorname{Hom}(\Omega^i(M),N \otimes_R C) \ar[r] & \operatorname{Ext}^1_R(\Omega^{i-1}(M),N \otimes_R C) \ar[r] & 0.
}}
\]
where the second row is exact. As to the first one, since $\Ext^1_R(\Omega^{i-1}(M), N)\cong \Ext^i_R(M, N)$, by assumption,  \cite[1.9(a)]{HomologicalAspectsOfSemidualizingModules} and Lemma \ref{lemswger}, all $R$-modules of this exact sequence belong to $\mathcal{A}_C(R)$. Therefore, the first row turns out to be exact as well by \cite[Lemma 3.1.12]{SemidualizingModules}.

Additionally, by \cite[Lemma 5.1]{AuslanderReitenCinjetivadimensionVanishingofExt}, the vertical maps are isomorphisms.  The five lemma settles the desired isomorphism. %Then, from the five lemma, we assure the existence of a natural isomorphism \linebreak  $\Ext_R^1( \Omega^{i-1}(M), N) \otimes_R C \cong \Ext_R^1(\Omega^{i-1}(M), N \otimes_R C)$, from which we conclude the result.
\end{proof}

Recall that the classes of modules of finite projective dimension, finite injective dimension, finite Gorenstein dimension, and finite Gorenstein injective dimension are closed under short exact sequences; see \cite[Corollaries 8.1.9 and 8.2.9, Proposition 9.2.15]{DerivedCategoryMethodsInCommutativeAlgebra} and \cite[Corollary 1.2.9]{GorensteinDimensions}.
With these facts in mind, together with the results established earlier in this section, we are now able to prove our main theorem.

\begin{theorem}\label{sqmm}
    Assume $R$ is Cohen-Macaulay with canonical module $\omega$, and let $n$ be a non-negative integer. Let $M$ and $N$ be $R$-modules such that $\operatorname{H-dim}_R(N)<\infty$.  Then the following conditions are equivalent:
    \begin{enumerate}[\rm(1)]
        \item $\operatorname{H-dim}_R(\Ext_R^i(M,N))<\infty$ for all $0\leq i \leq n$.
        \item $\operatorname{Hid}_R(\Ext_R^i(M, N \otimes_R \omega))<\infty$ for all $0\leq i \leq n$.
    \end{enumerate}
   In any case, we have natural isomorphisms $$\Ext_R^i(M,N) \otimes_R \omega \cong \Ext_R^i(M, N \otimes_R \omega)$$
    for all $0\leq i \leq n$. 
\end{theorem}

\begin{proof} Once the equivalence is proved, all the isomorphisms follow from Theorem \ref{Cdualt}, Proposition \ref{GdimCrs}, and item (1).

Now, we prove $(1) \Rightarrow (2)$. Suppose that $\operatorname{H-dim}_R(\Ext_R^i(M,N))<\infty$ for all $0\leq i \leq n$. Given $0\leq i \leq n$, we must show that $\operatorname{Hid}_R(\Ext_R^i(M,N \otimes_R \omega))<\infty$. The case $i=0$ follows from Proposition \ref{lemx}. Now, suppose $i>0$.   Consider the exact sequence $0 \stackrel{~}{\longrightarrow} \Omega^i(M) \stackrel{~}{\longrightarrow} R^{r_i} \stackrel{~}{\longrightarrow} \Omega^{i-1}(M)\longrightarrow 0.$ 
Then it induces an exact sequence 
\begin{equation*}\label{dfa}
    \scriptscriptstyle {0 \longrightarrow \Hom_R(\Omega^{i-1}(M), N\otimes_R \omega) \longrightarrow \Hom_R(R^{r_i}, N \otimes_R \omega ) \longrightarrow \Hom_R(\Omega^i(M),N \otimes_R \omega) \longrightarrow \Ext_R^1(\Omega^{i-1}(M), N \otimes_R \omega) \longrightarrow 0}.\end{equation*}  
By Proposition \ref{GdimCrs}(4), $\operatorname{Hid}_R(N \otimes_R \omega)<\infty$. Additionally, by Lemma \ref{lemswger}, \linebreak $\operatorname{H-dim}_R(\Hom_R(\Omega^j(M), N))<\infty$ for  $j=i-1,i$, so by Proposition \ref{lemx},   \linebreak $\operatorname{Hid}_R(\Hom_R(\Omega^{j}(M), N \otimes_R \omega))<\infty$ for $j=i-1,i$. Thus, the first three terms of the exact sequence above have finite $\operatorname{Hid},$ and therefore $\operatorname{Hid}_R(\Ext_R^1(\Omega^{i-1}(M), N \otimes_R \omega))<\infty$. The result follows from the fact that \ $\Ext_R^i(M,N \otimes_R \omega)  \cong \Ext_R^{1}(\Omega^{i-1}(M), N \otimes_R \omega)$. 

As to $(2)\Rightarrow(1)$, we follow the same lines as before by considering the exact sequence \begin{equation*}
    \scriptstyle {0 \longrightarrow \Hom_R(\Omega^{i-1}(M), N) \longrightarrow \Hom_R(R^{r_i}, N) \longrightarrow \Hom_R(\Omega^i(M),N) \longrightarrow \Ext_R^1(\Omega^{i-1}(M), N) \longrightarrow 0} \end{equation*}  
    for $0<i\leq n$. 
\end{proof}

The following corollary contains a generalization of \cite[Corollary 30]{HomologicalDimensionsTheGorensteinPropertyAndSpecialCasesOfSomeConjectures} and \cite[Theorem 4.2$(3) \Rightarrow (2)$]{OnModulesWhoseDualIsOfFiniteGorensteinDimension}.

\begin{corollary}\label{wqq11} \label{daso}
    Assume $R$ is Cohen-Macaulay local, and let $M$ and $N$ be non-zero $R$-modules such that $\operatorname{pd}_R(N)<\infty$. If $\operatorname{H-dim}_R(\Ext_R^i(M,N))<\infty$ for all $0\leq i \leq t-\operatorname{depth} (M)$, then $\operatorname{H-dim}_R(M)<\infty$. In addition, if $\Ext_R^i(M,N)=0$ for all $1\leq i \leq t-\depth (M)$, then $\operatorname{H-dim}_R (M)=0$. 
\end{corollary}
\begin{proof}
    We may assume that $R$ is complete, and hence $R$ has a canonical module $\omega$.  Since $\operatorname{pd}_R(N)<\infty$, by Proposition \ref{GdimCrs}, we have $\operatorname{id}_R(N  \otimes_R \omega)<\infty$ and, by Theorem \ref{sqmm}, $\operatorname{Hid}_R(\Ext_R^i(M, N \otimes_R \omega))<\infty$ for all $0\leq i\leq t-\depth (M)$. Then it follows from Corollary \ref{sqly} that $\operatorname{H-dim}_R(M)<\infty$.

    Now, by the vanishing assumption, as $\operatorname{pd}_R(N)<\infty$, we have that  $\Ext_R^i(M,N)=0$ for all $i>t-\depth (M)$ by \cite[Theorem 4.13]{Stablemoduletheory}, and that 
    $\operatorname{H-dim}_R(M)=\sup\{i \geq 0: \Ext_R^i(M,N)\not=0\}$ by \cite[Theorem 2.6(ii)]{VanishingOfCohomologyOverCompleteIntersectionRings}. Therefore, $\operatorname{H-dim}_R(M)=0$. 
\end{proof}

Combining Corollary \ref{SPA11} with Corollary \ref{wqq11}, we have the following byproduct.

\begin{corollary}\label{qza}
    Assume $R$ is Cohen-Macaulay, and let $M$ and $N$ be non-zero $R$-modules. Suppose that $\operatorname{pd}_R(\Ext_R^i(M,N))<\infty$ for all $0\leq i \leq t-\depth (M)$.  Then $\operatorname{pd}_R(M)<\infty$ if and only if $\operatorname{pd}_R(N)<\infty$.
\end{corollary}

Next, we explore the consequences of Theorem \ref{sqmm} when one considers $N=R$. Recall that the {\it type} of an $R$-module $M$ is defined to be the dimension of the $k$-vector space $\Ext^r_R(k, M)$, where $r = \depth (M)$, and denoted by $\type(M)$.

\begin{corollary}
    \label{rfg}
    Assume $R$ is Cohen-Macaulay with a canonical module $\omega$, and let $n$ be a non-negative integer. Let $M$ be an $R$-module such that $\operatorname{G-dim}_R(\Ext_R^i(M,R))<\infty$ for all $0\leq i \leq n$. Then for each $0\leq i \leq n$, we have that $\Ext_R^i(M,R) \otimes_R \omega \cong \Ext_R^i(M, \omega)$, $\operatorname{Tor}^R_1(\operatorname{Tr}_R(\Omega^i(M)), \omega)=0$ and $\mu(\Ext^i_R(M,\omega))=\mu(\Ext^i_R(M,R))\type(R)$.  
\end{corollary}
\begin{proof}
  The isomorphisms follow immediately from Theorem \ref{sqmm}. Since such isomorphisms are induced by the natural map $\Phi_{M,N}$ in Proposition \ref{lemx}, it follows from \cite[Theorem 2.8(b)]{Stablemoduletheory} that $\operatorname{Tor}^R_1(\operatorname{Tr}_R(\Omega^i(M)), \omega)=0$ for all $0 \leq i \leq n$.

 Now, by taking the tensor product by the residue field of $R$ and counting dimensions, the required equality holds from \cite[Proposition 3.3.11(c)]{bruns}. 
\end{proof}

\begin{corollary}\label{sqrqqw}
    Assume $R$ is Cohen-Macaulay, and let $M$ be an $R$-module of depth $r$ and dimension $s$. Suppose that $\operatorname{H-dim}_R(\Ext_R^i(M,R))<\infty$ for all $t-s\leq i \leq t-r$. Then $\operatorname{H-dim}_R (M)<\infty$. Moreover, $\operatorname{dim}_R(\Ext_R^i(M, R))\leq t-i$ for all $i\neq t-s$ and $\dim_R(\Ext^{t-s}_R(M, R))=s$. 
\end{corollary}

\begin{proof} We may assume that $R$ is complete and therefore $R$ admits a canonical module $\omega$. By \cite[Corollary 2.1.4]{bruns}, $\operatorname{grade}(M)=t-s$, so $\Ext_R^i(M,R)=0$ for all $i<t-s$ and $\Ext_R^{t-s}(M,R)\not=0$. Then by Corollary \ref{wqq11}, $\operatorname{H-dim}_R (M)<\infty$.

On the other hand, since  $\Ext_R^i(M, R)=0$ for $i>t-r$, one can see by assumption that $\operatorname{G-dim}_R(\Ext_R^i(M,R))<\infty$ for all $i\geq 0$. Thus, given $i\geq 0$,  by Corollary \ref{rfg}, $ \Ext_R^i(M,R) \otimes_R \omega \cong \Ext_R^i(M, \omega).$
    It follows that the modules $\Ext^i_R(M, R)$ and $\Ext^i_R(M, \omega)$ have the same support and so $\dim_R(\Ext_R^i(M,R))=\dim_R(\Ext_R^i(M, \omega))$. In this way, the last assertion  follows from \cite[Lemma 1.9 a) and c)]{Sch98}.
\end{proof}

\begin{remark}\label{rmkGhoshP}
    Ghosh and Puthenpurakal proved in \cite[Theorem 3.5]{GorensteinringsviahomologicaldimensionsandsymmetryinvanishingofExtandTatecohomology}
that if \(M\) is a finitely generated \(R\)-module that is \(\operatorname{G}\)-perfect
(i.e., \(\operatorname{grade}(M)=\operatorname{G\text{-}dim}_R(M)\)), then $\type(M)=\mu(\Ext^{t-r}_R(M,R))\type(R)$.
A more general formula was later obtained in \cite[Theorem (A)]{NumericalApectsofComplexesOfFiniteHomologicalDimensions}. As an application of Corollary~\ref{rfg}, one can prove the Ghosh--Puthenpurakal formula
to finitely generated $R$-modules $M$ over a Cohen--Macaulay ring $R$satisfying
$\operatorname{G\text{-}dim}_R(\Ext_R^i(M,R))<\infty$ for all $t-s\le i\le t-r$.
Indeed, as in the proof of Corollary~\ref{sqrqqw}, this finiteness condition holds for all $i\geq 0$.
Hence Corollary~\ref{rfg} yields $\mu(\Ext_R^i(M,\omega))=\mu(\Ext_R^i(M,R))\type(R)$, and the formula follows from \cite[Theorem 3.1]{FH}.
\end{remark}

\section{Some questions}\label{sectionsquestions}
In this section, we discuss some natural questions motivated by the results established earlier. Our primary question stems from the challenge of extending Corollary \ref{wqq11} to the non-Cohen-Macaulay setting.

\begin{question}\label{12qaz}
     Let $M$ and $N$ be non-zero $R$-modules such that $\operatorname{pd}_R(N)<\infty$. Do the following facts hold?
    \begin{enumerate}[\rm(1)]
        \item If $\operatorname{H-dim}_R(\Ext_R^i(M,N))<\infty$ for all $0\leq i \leq t$, then $\operatorname{H-dim}_R(M)<\infty$.
        \item If $\operatorname{H-dim}_R(\Hom_R(M,N))<\infty$ and $\Ext_R^{1\leq i \leq t}(M,N)=0$, then $\operatorname{H-dim}_R (M)=~0$.
    \end{enumerate}
\end{question}

The following proposition partially answers Question \ref{12qaz} when $\operatorname{H\text{-}dim}_R = \operatorname{pd}_R$.

\begin{proposition} \label{newprop}
    Let $M$ and $N$ be non-zero $R$-modules. Let $s$ be a non-negative integer such that $\operatorname{pd}_R(\Ext_R^i(M,N))<\infty$ for all $0\leq i \leq s$.  Then the equivalence $$\operatorname{pd}_R(M)<\infty \Longleftrightarrow \operatorname{pd}_R(N)<\infty$$
    holds in each one of the following cases:
    \begin{enumerate}[\rm (1)]
    \item If $s=t$ and $\Ext_R^i(M,R)=0$ for all $t+1\leq i \leq 2t+1$. 
        \item  $M$ has finite Gorenstein dimension equal to $s$.
    \end{enumerate}
\end{proposition}

\begin{proof}
In both cases, the implication $(\Rightarrow)$ follows from Corollary \ref{SPA11}. Thus, it remains to prove the implication $(\Leftarrow)$.

    (1) Suppose that  $\operatorname{pd}_R(N)<\infty$. By Lemma \ref{lemswger}, we have \linebreak $\operatorname{pd}_R(\Hom_R(\Omega^t(M), N))<\infty$.  Moreover,  the vanishing condition shows that \linebreak $\Ext_R^i(\Omega^t(M), R)=0$ for all $1\leq i \leq t$. Therefore, by  \cite[Proposition 19]{HomologicalDimensionsTheGorensteinPropertyAndSpecialCasesOfSomeConjectures}, $\operatorname{pd}_R(\Omega^t (M)^\ast)<\infty$. Then by \cite[Proposition 13]{HomologicalDimensionsTheGorensteinPropertyAndSpecialCasesOfSomeConjectures}, $\Omega^t(M)$ is free. Thus, $\operatorname{pd}_R(M)<\infty$. 

    (2) Since $\operatorname{pd}_R(N)<\infty$, then $\Ext_R^i(M,N)=0$ for all $i>\operatorname{G-dim}_R(M)=\depth(R)-\depth(M)$. Then by assumption, we can see that $\operatorname{pd}_R(\Ext_R^i(M,N))<\infty$ for all $0\leq i \leq t$. Therefore, the assertion follows from item (1).
    \end{proof}

%\begin{remark}
%It should be noticed that in Proposition \ref{newprop} the implication $\Rightarrow$ only requires  that $\operatorname{pd}_R(\Ext_R^i(M,N))<\infty$ for $0\leq i \leq t-\operatorname{depth} (M)$.
%\end{remark}

For $N=R$, an answer to Question \ref{12qaz}(2) is given by \cite[Proposition 13]{HomologicalDimensionsTheGorensteinPropertyAndSpecialCasesOfSomeConjectures}. Next, we present an answer to Question \ref{12qaz}(1) that strengthens this proposition. For that,  we say that an $R$-module $M$ satisfies  $(\widetilde{S}_n)$ if $\operatorname{depth}_{R_\mathfrak{p}}(M_\mathfrak{p})\geq \min\{n, \depth_{R_\mathfrak{p}}(R_\mathfrak{p})\}$ for all prime ideals $\mathfrak{p}$ of $R$. 

%\cite[Proposition 13]{HomologicalDimensionsTheGorensteinPropertyAndSpecialCasesOfSomeConjectures} for projective dimension and Gorenstein dimension. For that,  we say that an $R$-module $M$ satisfies  $(\widetilde{S}_n)$ if $\operatorname{depth}_{R_\mathfrak{p}}(M_\mathfrak{p})\geq \min\{n, \depth_{R_\mathfrak{p}}(R_\mathfrak{p})\}$ for all prime ideals $\mathfrak{p}$ of $R$. 
\begin{proposition}\label{eqee}
Let $M$ be an $R$-module satisfying $(\widetilde{S}_n)$ with $n\leq t$ a non-negative integer. Suppose that  $\operatorname{H-dim}_R(\Ext_R^i(M,R))<\infty$ for all $0\leq i \leq t-n$. Then $\Hdim_R(M)\leq t-n$. 
\end{proposition}
\begin{proof} Let $s=t-n$. Then by Lemma \ref{lemswger}, $\operatorname{H-dim}_R(\Omega^s(M)^\ast)<\infty$. Since $\operatorname{H-dim}_R(M^\ast)<\infty$ and $M$ satisfies $(\widetilde{S}_n)$, an application of the depth lemma shows that  $\Omega^s(M)$ satisfies $(\widetilde{S}_{t})$. It follows from \cite[Proposition 13]{HomologicalDimensionsTheGorensteinPropertyAndSpecialCasesOfSomeConjectures} that $\operatorname{H-dim}_R(\Omega^s(M))=0$. Hence, by \cite[Theorem 8.7(3)]{HomologicalDimensionsAndRelatedInvariantsOfModulesOverLocalRings}, $\operatorname{H-dim}_R(M)\leq s=t-n$.
\end{proof}

%The reason of this is that there are no non-zero $R$-modules $N$ with positive grade satisfying $(\widetilde{S}_1)$. The proof of this fact is as follows. Suppose that $N$ satisfies $(\widetilde{S}_1)$ and that $N^\ast=0$. Then, $\operatorname{pd}_R(\operatorname{Tr}_R(N))\leq 1$.  As $N$ satisfies $(\widetilde{S}_1)$, by \cite[Theorem 5.8]{HomologicalDimensionsOfRigidModules}, we have that $\Ext_R^1(\operatorname{Tr}_R(N),R)=0$. Thus, $\operatorname{Tr}_R(N)$ is free and hence $N$ as well. But since $N^\ast=0,$ then $N=0$. 

\begin{corollary}\label{seq88a}
    Let $M$ be an $R$-module. If $\operatorname{H-dim}_R(\Ext_R^i(M,R))<\infty$ for all $0\leq i \leq t$, then $\operatorname{H-dim}_R(M)<\infty$.
\end{corollary}

Note that Proposition \ref{eqee} and Corollary \ref{seq88a} give some answers to \cite[Question 4.11]{OnModulesWhoseDualIsOfFiniteGorensteinDimension}. As an application of Corollary \ref{seq88a},  we can improve \cite[Proposition 3.3]{ARCForModulesWhose(Self)DualHasFiniteCompleteIntersectionDimensionv2}. %as follows: 
\begin{corollary}\label{sedsm}\label{caj}
  Let $M$ be a non-zero $R$-module. Assume further that there exists a non-negative integer $h$ such that $\operatorname{H-dim}_R(\Ext_R^h(M,R))<\infty$ and $\operatorname{Ext}_R^i(M,R)=0$ for all $i \in \{0, \ldots, t\}-\{h\}$, then $\operatorname{H-dim}_R(M)=h$.
\end{corollary}

\begin{proof}
    From the assumption we see that $\operatorname{H-dim}_R(\Ext_R^i(M,R))<\infty$ for all $0 \leq i \leq t$. Then by Corollary \ref{seq88a}, $\operatorname{H-dim}_R(M)<\infty$. Since $\operatorname{H-dim}_R(M)=\sup\{i\geq 0: \Ext_R^i(M,R)\not=0\}$ (see, for instance \cite[Lemma 2.3(c)]{GorensteinDimensionAndTorsionOfModulesOverCommutativeNoetherianRings} together with \cite[Theorem 1.4]{Completeintersectiondimension}) and $\operatorname{H-dim}_R(M)\leq t$, we see from the vanishing condition that $\operatorname{H-dim}_R(M)=~h$.  
\end{proof}

The reader can note that Corollary \ref{daso}, which motivated Question \ref{12qaz}, was obtained using Theorem \ref{sqmm}, which was derived from Theorem \ref{Cdualt}. Thus, in view of  the Foxby equivalence (\cite[Theorem 3.3.2]{GorensteinDimensions}), we pose the following question:

\begin{question} It is possible to obtain a variation of Proposition \ref{lemx} or Theorem   \ref{Cdualt}  involving a dualizing complex $\boldsymbol{D}$ instead of $\omega$ or $C$?
\end{question}

Now, motivated by Corollary \ref{qza}, we pose the following:
\begin{question}\label{qaz}
    Let $M$ and $N$ be non-zero $R$-modules such that $\operatorname{pd}_R(\Ext_R^i(M,N))<\infty$ for all $0\leq i \leq t$ (or all $i \geq 0$). Then does the equivalence 
    $$\operatorname{id}_R(M)<\infty \Longleftrightarrow \operatorname{id}_R(N)<\infty$$ hold?
\end{question}
Note that the implication $(\Leftarrow)$ holds by Theorem \ref{q221}.  On the other hand, the other implication has a positive answer if $M$ is maximal Cohen-Macaulay. Indeed, by the Bass conjecture, $R$ is Cohen-Macaulay, and we may assume (passing to completion if necessary) that $R$ admits a canonical module. Then $M\cong \omega^n$ for some $n>0$ and $\operatorname{pd}_R(\Hom_R(\omega, N))<\infty$, obtaining thus the desired conclusion from Proposition \ref{GdimCrs}(3).

Motivated by Corollary \ref{newGor} and \cite[Corollary 5.13]{OnModulesWhoseDualIsOfFiniteGorensteinDimension}, we close this section by asking the following.

\begin{question}
Assume $R$ is Cohen-Macaulay with canonical module $\omega$, and let $\_^\dagger$ stand for $\Hom_R(\_,\omega)$. Does any of the conditions $\operatorname{Gid}_R(\omega^*)<\infty$ or $\operatorname{G-dim}_R((\omega^*)^\dagger)<\infty$  imply that $R$ must be Gorenstein?
\end{question}

\subsection*{Addendum} After the first version of this paper, recently
Questions \ref{12qaz} and \ref{qaz} were affirmatively answered by Kaito Kimura in \cite{FinitenessOfHomologicalDimensionOfExtModules}. The technique in his paper differs from the one we used in 
Theorem \ref{wqq11}.

\section{Examples}\label{sectionexamples}

We present some examples concerning Ext modules with finite homological dimension. 

\begin{example}
    Assume $R$ is a local ring and let $\underline{x}=x_1,\ldots,x_n$ be a regular sequence. Let $N$ be an $R$-module such that $\underline{x}N=0$. Set $M=R/(\underline{x})$. Since the minimal free resolution of $M$ is the Koszul complex of the sequence $\underline{x}$, we have
    $$\Ext^i_R(M,N)\cong N^{\binom{n}{i}}, \quad\forall \ i=0,\ldots,n.$$
    Thus, $\Hdim_R(\Ext^i_R(M,N))<\infty$ for all $0\leq i \leq n$ if and only if $\Hdim_R(N)<\infty.$ Similarly, $\operatorname{Hid}_R(\Ext^i_R(M,N))<\infty$ for all $0\leq i \leq n$ if and only if $\operatorname{Hid}_R(N)<\infty.$

    Now, set $L=\oplus_{i=1}^nR/(x_1,\ldots,x_i)$ and let $X$ be a finitely generated $R$-module such that $\underline{x}$ is $X$-regular. For all $i=1,\ldots,n$, by the Koszul duality we must have $$\Ext^i_R(L,X)\cong\Ext^i_R(R/(x_1,\ldots,x_i),X)\cong X/(x_1,\ldots,x_i)X.$$

    Thus $\operatorname{H-dim}_R(X)<\infty$ if and only if $\operatorname{H-dim}_R( \Ext_R^i(L,X))<\infty$ for all $0\leq i \leq n$. In particular,   $\Hdim_R(\Ext^i_R(L,R))=i$, by the Auslander-Buchsbaum and Auslander-Bridger formulas.
\end{example}

An $R$-module $M$ is said to be $\operatorname{H-}$\textit{perfect} if $\operatorname{H-dim}_R(M)<\infty$ and $\operatorname{grade}(M)=\operatorname{H-dim}_R(M)$.
\begin{example}
    Let $M_1,\ldots, M_n$ be H-perfect $R$-modules of grade $g_1,\ldots, g_n$, respectively. Then each  $\Ext_R^{g_i}(M_i, R)$ have $\operatorname{H-dim}$ equal $g_i$ by \cite[Theorem 6.3]{AbsoluteRelativeAndTateCohomologyOfModulesOfFiniteGorensteinDimension}. Set $M=\bigoplus_{i=1}^n M_i$. Thus, for  $j\not=g_1, \ldots, g_n$, the $R$-module $\Ext_R^j(M,R)$ is zero, while that $\operatorname{H-dim}_R(\Ext_R^{g_i}(M,R))=g_i$ for each $1\leq i \leq n$. 
\end{example}

 An $R$-module $M$ is $p$-\textit{spherical} if $\operatorname{pd}_R(M)\leq p$ and $\Ext_R^i
(M, R) =0$ for all $1\leq i \leq p-1$. 

\begin{example} Based on the proof of \cite[Proposition 6]{anneauxdegorensteinettorsionenalgebrecommutative} we can construct spherical non-perfect modules as follows. We repeat the proof for convenience. Let $M$ be an $R$-module with positive grade $n$, and let $F_\bullet$ be a free resolution of $M$. Consider the truncation
    $$\xymatrix@=1em{
    0\ar[r] & F_n\ar[r]^\phi & F_{n-1}\ar[r] & \cdots\ar[r] & F_2\ar[r] & F_1\ar[r] & F_0\ar[r] & M\ar[r] & 0.
    }$$
    Set $N:=\operatorname{coker}(\phi^*)$. Since $\grade (M)=n>0$, the complex
    $$\xymatrix@=1em{
    0\ar[r] & F_0^*\ar[r] & F_1^*\ar[r] & F_2^*\ar[r] & \cdots\ar[r] & F_{n-1}^*\ar[r]^{\phi^*} & F_n^*\ar[r] & N\ar[r] & 0
    }$$
    is a free resolution of $N$ and
    $$\xymatrix@=1em{
    0\ar[r] & N^*\ar[r] & F_n\ar[r]^{\phi} & F_{n-1}\ar[r] & \cdots\ar[r] & F_2\ar[r] & F_1\ar[r] & F_0\ar[r] & M\ar[r] & 0
    }$$
    is exact. This means that $N$ is $n$-spherical as well as $N^*$ is a $(n+1)$th syzygy of $M$ (in the sense of \cite{GorensteinDimensionAndTorsionOfModulesOverCommutativeNoetherianRings}), and also $\Ext^n_R(N,R)\cong M$. It follows that whenever $M$ is such that $\pd_R(M)>\grade(M)=n>0$, there exists a $n$-spherical non-perfect $R$-module $N$ such that $\Ext^n_R(N,R)\cong M$. In particular, if $\operatorname{H-dim}_R(M)<\infty$ (resp. $\operatorname{Hid}_R(M)<\infty)$ then $\operatorname{H-dim}_R(\Ext_R^n(N,R))<\infty$ (resp. $\operatorname{Hid}_R(\Ext_R^n(N,R))<\infty$).
\end{example}

\begin{example}
   Let $k$ be a field and consider $R=k[[x,y]]/(xy)$. Then $M=R/(x)$ has infinite free resolution
    $$\xymatrix@=1em{\cdots\ar[r] & R\ar[r]^x & R\ar[r]^y & R\ar[r]^x & R\ar[r] & M\ar[r] & 0}$$
   so that $\Ext^{2i+1}_R(M,M)\cong M$ for all $i$, which means that $\operatorname{G-dim}_R(\Ext^{2i+1}_R(M,M))=0$ and $\pd_R(\Ext^{2i+1}_R(M,M))=\infty$.
\end{example}

\bibliographystyle{plain}
\bibliography{references.bib}

\begin{thebibliography}{10}

\bibitem{Stablemoduletheory}
M.~Auslander and M.~Bridger.
\newblock {\em Stable module theory}.
\newblock Mem. of the AMS, No. 94. American Mathematical Society, Providence, R.I., 1969.

\bibitem{Completeintersectiondimension}
L.~Avramov, V.~Gasharov, and I.~Peeva.
\newblock Complete {I}ntersection dimension.
\newblock {\em Publications Math\'ematiques de l'IH\'ES}, 86:67--114, 1997.

\bibitem{AbsoluteRelativeAndTateCohomologyOfModulesOfFiniteGorensteinDimension}
L.~Avramov and A.~Martsinkovsky.
\newblock Absolute, relative, and {T}ate cohomology of modules of finite {G}orenstein dimension.
\newblock {\em Proc. Lond. Math. Soc.}, 85, 11 2000.

\bibitem{HomologicalDimensionsAndRelatedInvariantsOfModulesOverLocalRings}
Luchezar~L. Avramov.
\newblock Homological dimensions and related invariants of modules over local rings.
\newblock In {\em Representations of algebra. {V}ol. {I}, {II}}, pages 1--39. Beijing Norm. Univ. Press, Beijing, 2002.

\bibitem{bruns}
W.~Bruns and J.~Herzog.
\newblock {\em Cohen-{M}acaulay Rings}.
\newblock Cambridge University Press, New York, 1998.

\bibitem{GorensteinDimensions}
L.~W. Christensen.
\newblock {\em Gorenstein Dimensions}, volume 1747.
\newblock Springer, 2000.

\bibitem{Semi-dualizingComplexesAndTheirAuslanderCategores}
L.~W. Christensen.
\newblock Semi-dualizing complexes and their {A}uslander categories.
\newblock {\em Trans. Amer. Math. Soc.}, 353:1839--1883, 2001.

\bibitem{HyperhomologicalAlgebra}
L.~W. Christensen and H.-B. Foxby.
\newblock Hyperhomological algebra with applications to commutative.
\newblock Avaliable at https://www.math.ttu.edu/~lchriste/download/918-final.pdf, 2006.

\bibitem{OnGorensteinProjectiveInjectiveAndFlatDimensionsAFunctorialDescriptionWithApplications}
L.~W. Christensen, A.~Frankild, and H.~Holm.
\newblock On {G}orenstein projective, injective and flat dimensions—a functorial description with applications.
\newblock {\em J. Algebra}, 302(1):231--279, 2006.

\bibitem{DerivedCategoryMethodsInCommutativeAlgebra}
L.W. Christensen, H.-B. Foxby, and H.~Holm.
\newblock {\em Derived Category Methods in Commutative Algebra}.
\newblock Springer Monographs in Mathematics Series. Springer, 2024.

\bibitem{HomologicalDimensionsTheGorensteinPropertyAndSpecialCasesOfSomeConjectures}
S.~Dey, R.~Holanda, and C.~B. Miranda-Neto.
\newblock Homological dimensions, the {G}orenstein property, and special cases of some conjectures.
\newblock {\em arXiv 2405.00152v1}, 2024.

\bibitem{FH}
T.~Fiel and R.~Holanda.
\newblock Bass and {Betti} numbers of a module and its deficiency modules.
\newblock {\em J. Commut. Algebra}, 16(2):197--211, 2024.

\bibitem{onthemuiinaminimalinjectiveresolution}
H.-B. Foxby.
\newblock On the $\mu^i$ in a minimal injective resolution.
\newblock {\em Math. Scand.}, 29(2):175--186, 1971.

\bibitem{isomorphismsbetweencomplexeswithapplicationstothehomologicaltheoryofmodules}
H.-B. Foxby.
\newblock Isomorphisms between complexes with applications to the homological theory of modules.
\newblock {\em Math. Scand.}, 40:5--19, 1977.

\bibitem{GorensteinringsviahomologicaldimensionsandsymmetryinvanishingofExtandTatecohomology}
D.~Ghosh and T.~J. Puthenpurakal.
\newblock Gorenstein rings via homological dimensions, and symmetry in vanishing of {E}xt and {T}ate cohomology.
\newblock {\em Algebr. Represent. Theory}, 27(1):639--653, 2024.

\bibitem{ARCForModulesWhose(Self)DualHasFiniteCompleteIntersectionDimensionv2}
D.~Ghosh and M.~Samanta.
\newblock Auslander-{R}eiten conjecture for modules whose (self) dual has finite {C}omplete {I}ntersection dimension.
\newblock {\em arXiv 2405.01497v2}, 2024.

\bibitem{injectivedimensiontakahashi}
D.~Ghosh and R.~Takahashi.
\newblock Auslander-{R}eiten conjecture and finite injective dimension of {$\operatorname{Hom}$}.
\newblock {\em Kyoto J. Math.}, 64(1):229 -- 243, 2024.

\bibitem{VanishingOf(Co)homologyfrenessCriteriaAndTheAuslanderReitenConjectureforCohen-MacaulayBurch}
R.~Holanda and C.~B. Miranda-Neto.
\newblock Vanishing of (co)homology, freeness criteria, and the {A}uslander-{R}eiten conjecture for {C}ohen-{M}acaulay {B}urch rings.
\newblock {\em arXiv 2212.05521}, 2022.

\bibitem{GdimensionOverLocalHomomorphisms}
S.~Iyengar and S.~Sather-Wagstaff.
\newblock {G-dimension over local homomorphisms. Applications to the Frobenius endomorphism}.
\newblock {\em Illinois J. Math.}, 48(1):241 -- 272, 2004.

\bibitem{FinitenessOfHomologicalDimensionOfExtModules}
K.~Kimura.
\newblock Finiteness of homological dimensions of {Ext} modules.
\newblock {\em arXiv preprint arXiv:2508.12843v2}, 2025.

\bibitem{anneauxdegorensteinettorsionenalgebrecommutative}
M.~Mangeney, C.~Peskine, and L.~Szpiro.
\newblock Anneaux de {Gorenstein,} et torsion en alg\`ebre commutative.
\newblock {\em S\'eminaire Samuel. Alg\`ebre commutative}, 1:2--69, 1966-1967.

\bibitem{GorensteinDimensionAndTorsionOfModulesOverCommutativeNoetherianRings}
V.~Maşiek.
\newblock Gorenstein dimension and torsion of modules over commutative {N}oetherian rings.
\newblock {\em Commun. Algebra}, 28(12):5783--5811, 2000.

\bibitem{AuslanderReitenCinjetivadimensionVanishingofExt}
V.~D. Mendoza-Rubio and V.~H. Jorge-P{\'e}rez.
\newblock The {A}uslander–{R}eiten conjecture finite {C}-injective dimension of {Hom}, and vanishing of {Ext}.
\newblock {\em J. Commut. Algebra}, 17(2):153--168, 2025.

\bibitem{OnModulesWhoseDualIsOfFiniteGorensteinDimension}
V.D. Mendoza-Rubio and V.~H. Jorge-P{\'e}rez.
\newblock On modules whose dual is of finite {G}orenstein dimension.
\newblock {\em Collect. Math.}, 2025.

\bibitem{Dimensionprojectivefinietcohomologielocale}
C.~Peskine and L.~Szpiro.
\newblock Dimension projective finie et cohomologie locale.
\newblock {\em Publications Math\'ematiques de l'IH\'ES}, 42:47--119, 1973.

\bibitem{VanishingOfCohomologyOverCompleteIntersectionRings}
A.~Sadeghi.
\newblock Vanishing of cohomology over {C}omplete {I}ntersection rings.
\newblock {\em Glasg. Math. J.}, 57(2):445–455, 2015.

\bibitem{SemidualizingModules}
S.~Sather-Wagstaff.
\newblock Semidualizing modules.
\newblock Avaliable at \url{https://ssather.people.clemson.edu/DOCS/sdm.pdf}.

\bibitem{CompleteIntersectionDimensionsandFoxbyClasses}
S.~Sather-Wagstaff.
\newblock Complete {I}ntersection dimensions and {F}oxby classes.
\newblock {\em J. Pure Appl. Algebra}, 212(12):2594--2611, 2008.

\bibitem{Sch98}
P.~Schenzel.
\newblock On the use of local cohomology in algebra and geometry.
\newblock In {\em Six lectures on commutative algebra. Lectures presented at the summer school, Bellaterra, Spain, July 16--26, 1996}, pages 241--292. Basel: Birkh{\"a}user, 1998.

\bibitem{HomologicalAspectsOfSemidualizingModules}
R.~Takahashi and D.~White.
\newblock Homological aspects of semidualizing modules.
\newblock {\em Math. Scand.}, 106(1):5--22, 2010.

\bibitem{HomologicalDimensionsOfRigidModules}
M.~R. Zargar, O.~Celikbas, M.~Gheibi, and A.~Sadeghi.
\newblock {Homological dimensions of rigid modules}.
\newblock {\em Kyoto J. Math.}, 58(3):639 -- 669, 2018.

\bibitem{NumericalApectsofComplexesOfFiniteHomologicalDimensions}
M.~R. Zargar and M.~Gheibi.
\newblock Numerical aspects of complexes of finite homological dimensions.
\newblock {\em J. Commut. Algebra}, 3(16):353--362, 2024.

\end{thebibliography}
\end{document}